\documentclass[11pt]{amsart}
\usepackage{amsmath,cite}
\usepackage[colorlinks=true,urlcolor=blue,
citecolor=red,linkcolor=blue,linktocpage,pdfpagelabels,bookmarksnumbered,bookmarksopen]{hyperref}
\usepackage[english]{babel}

\usepackage[left=2.7cm,right=2.7cm,top=2.55cm,bottom=2.55cm]{geometry}

\numberwithin{equation}{section}
\newtheorem{theorem}{Theorem}[section]
\newtheorem{proposition}[theorem]{Proposition}
\newtheorem{lemma}[theorem]{Lemma}
\newtheorem{remark}[theorem]{Remark}
\newtheorem{corollary}[theorem]{Corollary}

\theoremstyle{definition}

\newcommand{\M}{{\mathbb M}}
\newcommand{\N}{{\mathbb N}}
\newcommand{\R}{{\mathbb R}}

\newcommand{\D}{{\mathbb D}}
\renewcommand{\S}{{\mathbb S}}
\newcommand{\dvg}{{\rm div}}

\newcommand{\eps}{\varepsilon}

\title[Global compactness for quasi-linear problems ]{Global compactness for a class of \\
quasi-linear elliptic problems}

\author[C.~Mercuri]{Carlo Mercuri}
\author[M.~Squassina]{Marco Squassina}

\address{Department of Mathematics and Computer Science
\newline\indent
Technische Universiteit Eindhoven
\newline\indent
Postbus 513,
5600 MB Eindhoven
\newline\indent
Holland}
\email{c.mercuri@tue.nl}

\address{Department of Computer Science
\newline\indent
University of Verona
\newline\indent
Strada Le Grazie 15, 37134 Verona
\newline\indent
Italy}
\email{marco.squassina@univr.it}

\thanks{Supported by Miur project: {\em ``Variational and Topological
Methods in the Study of Nonlinear Phenomena''}}


\begin{document}

\begin{abstract}
We prove a global compactness result for Palais-Smale sequences associated
with a class of quasi-linear elliptic equations on exterior domains.
\end{abstract}

\maketitle

\section{Introduction and main result}
Let $\Omega$ be a smooth domain of $\R^N$ with a bounded complement and $N>p>m>1$.
The main goal of this paper is to obtain a global compactness result for the Palais-Smale sequences of the
energy functional associated with the following quasi-linear elliptic equation
\begin{equation}\label{eq}
-\dvg(L_\xi(Du))-\dvg(M_\xi(u,Du)) + M_s(u,Du) + V (x)|u|^{p-2}u = g(u) \quad\text{in $\Omega$,}
\end{equation}
where $u\in W^{1,p}_0(\Omega)\cap D^{1,m}_0(\Omega)$, meant as the completion of the
space ${\mathcal D}(\Omega)$ of smooth functions with compact support,
with respect to the norm $\|u\|_{W^{1,p}(\Omega)\cap D^{1,m}(\Omega)}=\|u\|_p+\|u\|_m,$ having set
$\|u\|_p:=\|u\|_{W^{1,p}(\Omega)}$ and $\|u\|_m:=\|Du\|_{L^m(\Omega)}$.
We assume that $V$ is a continuous function on $\Omega$,
$$
\lim_{|x|\to \infty}V(x)=V_\infty
\,\,\quad
\text{and}
\,\,\quad
\inf_{x\in\Omega} V(x)=V_0>0.
$$
As known, lack of compactness may occur
due to the lack of compact embeddings for Sobolev spaces on
$\Omega$ and since the limiting equation on $\R^N$
\begin{equation}\label{eqlim}
	-\dvg(L_\xi(Du))-\dvg(M_\xi(u,Du)) + M_s(u,Du) + V_\infty|u|^{p-2}u = g(u) \quad\text{in $\R^N$},
\end{equation}
with $u\in W^{1,p}(\R^N)\cap D^{1,m}(\R^N)$, is invariant by translations. A particular case of \eqref{eq} is
\begin{equation}\label{eqm2}
-\Delta_pu-\dvg(a(u)|Du|^{m-2}Du) + \frac{1}{m}a'(u)|Du|^m+ V (x)|u|^{p-2}u =|u|^{\sigma-2}u \quad\text{in $\Omega$,}
\end{equation}
where $\Delta_p u:= \dvg (|D u|^{p-2}D u),$ for a suitable function $a\in C^1(\R;\R^+)$, or the even simpler case where $a$ is constant, namely
\begin{equation}\label{eqm1}
-\Delta_pu-\Delta_m u+ V (x)|u|^{p-2}u =|u|^{\sigma-2}u \quad\text{in $\Omega$}.
\end{equation}

Since the pioneering work of Benci and Cerami \cite{bencicer}
dealing with the case $L(\xi)=|\xi|^2/2$ and $M(s,\xi)\equiv 0$, many papers have been written on this subject, see for instance
the bibliography of \cite{mercuriwillem}. Quite recently, in \cite{mercuriwillem},
the case $L(\xi)=|\xi|^p/p$ and $M(s,\xi)\equiv 0$ was investigated. The main point in the present contribution
is the fact that we allow, under suitable assumptions, a quasi-linear term $M(u,Du)$ depending on the unknown $u$ itself.
The typical tools exploited in \cite{bencicer,mercuriwillem}, in addition to the point-wise convergence of the gradients, are
some decomposition (splitting) results both for the energy functional and for the equation, along a given bounded
Palais-Smale sequence $(u_n)$. To this regard, the explicit dependence on $u$ in the term $M(u,Du)$ requires a rather careful analysis.
In particular, we can handle it for
$$
\nu|\xi|^m\leq M(s,\xi)\leq C|\xi|^m,\qquad
p-1 \leq m < p-1+p/N.
$$
The restriction on $m$, together with the sign condition \eqref{ilsegnos} provides, thanks to the presence of $L,$
the needed a priori regularity on the weak limit of $(u_n),$  see Theorems \ref{energytot} and \ref{split2-A}.\\
Besides the aforementioned motivations, which are of mathematical interest,
it is worth pointing out that in recent years, some works have been devoted to quasi-linear operators
with double homogeneity, which arise from several problems of Mathematical Physics. For instance,
the reaction diffusion problem $u_t  = -\dvg(\D(u)Du) + \ell(x,u)$,
where $\D(u) = d_p|Du|^{p-2}+d_m|Du|^{m-2}$, $d_p>0$ and $d_m>0$,
admitting a rather wide range of applications in biophysics \cite{fife}, plasma physics \cite{wilh}
and in the study of chemical reactions \cite{aris}.
In this framework, $u$ typically describes a concentration and $\dvg(\D(u)Du)$
corresponds to the diffusion with a coefficient $\D(u)$, whereas $\ell(x,u)$
plays the r\v ole of reaction and relates to source and loss processes. We refer the interested reader
to \cite{motivaz} and to the reference therein. Furthermore, a model for elementary particles proposed
by Derrick \cite{derrick} yields to the study of standing wave solutions $\psi(x,t)=u(x)e^{{\rm i}\omega t}$
of the following nonlinear Schr\"odinger equation
$$
{\rm i}\psi_t+\Delta_2\psi-b(x)\psi+\Delta_p\psi-V(x)|\psi|^{p-2}\psi+|\psi|^{\sigma-2}\psi=0 \quad\text{in $\R^N$,}
$$
for which we refer the reader e.g.\ to \cite{bencifor}.

\vskip2pt
In order to state the first main result, assume $N>p>m\geq 2$ and
\begin{equation}
	\label{range}
p-1 \leq m < p-1+p/N,
\qquad
p<\sigma<p^*,
\end{equation}
and consider the $C^2$ functions $L:\R^N\to\R$ and $M:\R\times\R^N\to\R$ such that
both the functions $\xi\mapsto L(\xi)$ and $\xi\mapsto M(s,\xi)$ are strictly convex and
\begin{equation}
	\label{growth0}
\nu|\xi|^p\leq |L(\xi)|\leq C|\xi|^p,\quad |L_\xi(\xi)|\leq C|\xi|^{p-1},\quad |L_{\xi \xi}(\xi)| \leq C |\xi|^{p-2},
\end{equation}
for all $\xi\in\R^N$. Furthermore, we assume
\begin{align}
	\label{growths1}
 \nu|\xi|^m\leq M(s,\xi)|&\leq C|\xi|^m,
\quad\,\, |M_s(s,\xi)|\leq C|\xi|^{m},
\quad\,\,
|M_\xi(s,\xi)|\leq C|\xi|^{m-1}, \\
\label{growths2}
 |M_{ss}(s,\xi)|&\leq C|\xi|^m,
\quad\,\, |M_{s\xi}(s,\xi)|\leq C |\xi|^{m-1},
\quad\,\, |M_{\xi \xi}(s,\xi)|\leq C |\xi|^{m-2},
\end{align}
for all $(s,\xi)\in\R\times\R^N$ and that the sign condition (cf. \cite{squbook})
\begin{equation}
	\label{ilsegnos}
M_s(s,\xi)s\geq 0,
\end{equation}
holds for all $(s,\xi)\in\R\times\R^N$. Also, $G:\R\to\R$ is a $C^2$ function with $G'(s):=g(s)$ and
\begin{equation}
	\label{ggrow}
|G'(s)|\leq C|s|^{\sigma-1},\quad
|G''(s)|\leq C |s|^{\sigma -2},
\end{equation}
for all $s\in\R$. We define
\begin{equation}
	\label{jdef}
j(s,\xi):=L(\xi)+M(s,\xi)-G(s),
\end{equation}
and on $W^{1,p}_0(\Omega)\cap D^{1,m}_0(\Omega)$ with $\|u\|_{W^{1,p}(\Omega)\cap D^{1,m}(\Omega)}=\|u\|_p+\|u\|_m$ the functional
$$
\phi(u):=\int_{\Omega}j(u,Du)+\int_{\Omega} V(x)\frac{|u|^p}{p}.
$$
Finally, on $W^{1,p}(\R^N)\cap D^{1,m}(\R^N)$ with $\|u\|_{W^{1,p}(\R^N)\cap D^{1,m}(\R^N)}=\|u\|_p+\|u\|_m$ we define
$$
\phi_\infty(u):=\int_{\R^N}j(u,Du)+\int_{\R^N}V_\infty\frac{|u|^p}{p}.
$$
See Section~\ref{prelimsection} for some properties of the functionals $\phi$ and $\phi_\infty$.
\vskip3pt
\noindent
The first main global compactness type result is the following

\begin{theorem}
	\label{main}
Assume that \eqref{range}-\eqref{jdef} hold and
let $(u_n)\subset W^{1,p}_0(\Omega)\cap D^{1,m}_0(\Omega)$ be a bounded sequence such that
$$
\phi(u_n)\to c \quad \quad \phi'(u_n)\to 0 \quad \text{in $(W^{1,p}_0(\Omega)\cap D^{1,m}_0(\Omega))^*$}
$$
Then, up to a subsequence, there exists a
weak solution $v_0\in W^{1,p}_0(\Omega)\cap D^{1,m}_0(\Omega)$ of
$$
-\dvg(L_\xi(Du))-\dvg(M_\xi(u,Du)) + M_s(u,Du) + V (x)|u|^{p-2}u = g(u) \quad\text{in $\Omega$,}
$$
a finite sequence $\{v_1,...,v_k\}\subset  W^{1,p}(\R^N)\cap D^{1,m}(\R^N)$ of weak solutions of
$$
-\dvg(L_\xi(Du))-\dvg(M_\xi(u,Du)) + M_s(u,Du) + V_\infty|u|^{p-2}u = g(u) \quad \text{in $\R^N$}
$$
and $k$ sequences $(y^i_n)\subset\R^N$  satisfying
$$
|y^i_n|\to \infty,\quad |y^i_n-y^j_n|\to \infty,\quad i\neq j, \quad \text{as $n\to \infty$,}
$$
$$
\|u_n-v_0-\sum^k_{i=1}v_i ((\cdot-y^i_n)\|_{W^{1,p}(\R^N)\cap D^{1,m}(\R^N)}\to 0, \qquad\text{as $n\to \infty$,}
$$
$$
\|u_n\|_p^p\to \sum^k_{i=0}\|v_i\|^p_p,\qquad
\|u_n\|_m^m\to \sum^k_{i=0}\|v_i\|^m_m,\qquad\text{as $n\to \infty$,}
$$
as well as
$$
\phi(v_0)+\sum^k_{i=1}\phi_\infty (v_i)=c.
$$
\end{theorem}

\noindent
Let us now come to a statement for the cases $1<m\leq 2$ or $1<p\leq 2$. Let us define
\begin{align*}
 &\mathfrak{L}(\xi,h):=\frac{|L_\xi(\xi+h)-L_\xi(\xi)|}{|h|^{p-1}}, \qquad \text{if}\,\, 1<p<2,  \\
 & \mathfrak{G}(s,t):= \frac{|G'(s+t)-G'(s)|}{|t|^{\sigma-1}},\qquad \text{if}\,\, 1<\sigma<2,  \\
 &  \mathfrak{M}(s,\xi,h):=\frac{|M_\xi(s,\xi+h)-M_\xi(s,\xi)|}{|h|^{m-1}},\qquad \text{if}\,\, 1<m<2.
\end{align*}
If either $p<2,$ $\sigma<2$ or $m<2$, we shall weaken the twice differentiability
assumptions, by requiring $L_\xi\in C^1(\R^N\setminus\{0\})$,
$G'\in C^1(\R\setminus\{0\})$, $M_\xi \in C^1(\R\times (\R^N\setminus\{0\}))$,
$M_{s\xi}\in C^0(\R\times\R^N)$ and $M_{ss}\in C^0(\R\times\R^N)$.
Moreover we assume the same growth conditions for $L,M,G$ and
their derivatives, replacing only the growth assumptions for $L_{\xi\xi}, M_{\xi\xi},G''$ by the following hypotheses:
\begin{align}
	\label{ass-sub1}
 & \sup_{h\neq 0,\,\xi\in\R^N}\mathfrak{L}(\xi,h)<\infty, \\
\label{ass-sub2}
 & \sup_{t\neq 0,\, s\in\R} \mathfrak{G}(s,t)<\infty, \\
\label{ass-sub3}
 &  \sup_{h\neq 0,\, (s,\xi)\in\R\times\R^N} \mathfrak{M}(s,\xi,h)<\infty.
\end{align}
Conditions \eqref{ass-sub1}-\eqref{ass-sub2}, in some more concrete situations, follow immediately by homogeneity of
$L_\xi$ and $G'$ (see, for instance, \cite[Lemma 3.1]{mercuriwillem}).
Similarly, \eqref{ass-sub3} is satisfied for instance when $M$ is of the form $M(s,\xi)=a(s)\mu(\xi)$, being $a:\R\to\R^+$
a bounded function and $\mu:\R^N\to\R^+$ a $C^1$ strictly convex function such that $\mu_\xi$ is homogeneous of degree $m-1$.


\begin{theorem}
	\label{main2}
Under the additional assumptions \eqref{ass-sub1}-\eqref{ass-sub3} in the sub-quadratic cases, the assertion of Theorem~\ref{main} holds true.
\end{theorem}

\noindent
As a consequence of the above results we have the following compactness criterion.

\begin{corollary}\label{compat}
Assume \eqref{bounddc} below for some $\delta>0$ and $\mu>p$. Under the hypotheses of Theorem \ref{main} or \ref{main2}, if
$c<c^*,$ then $(u_n)$ is relatively compact in $W^{1,p}_0(\Omega)\cap D^{1,m}_0(\Omega)$ where
$$
c^*:=\min\left\{\frac{\delta}{\mu},\frac{\mu-p}{\mu p}V_\infty\right\}\left[\frac{\min\{\nu, V_\infty\}}{C_g S_{p,\sigma}}\right]^{\frac{p}{\sigma-p}},
$$
and $S_{p,\sigma}$ and $C_{g}$ are constants such that
$S_{p,\sigma}\|u\|^\sigma_p\geq \|u\|^\sigma_{L^\sigma(\R^N)}$
and $|g(s)|\leq C_{g}|s|^{\sigma-1}$.
\end{corollary}

\begin{remark}\rm
It would be interesting to get a global compactness result in the case $L=0$
and $p=m$, namely for the model case
\begin{equation}\label{eqm3}
-\dvg(a(u)|Du|^{m-2}Du) + \frac{1}{m}a'(u)|Du|^m+ V (x)|u|^{m-2}u =|u|^{\sigma-2}u \quad\text{in $\Omega$.}
\end{equation}
Notice that, even assuming $a'$ bounded, $a'(u)|Du|^m$ is merely in $L^1(\Omega)$
for $W^{1,m}_0(\Omega)$ distributional solutions. In general, in this setting, the splitting properties of the equation
are hard to formulate in a reasonable fashion.
\end{remark}

\begin{remark}\rm
The restriction of between $m$ and $p$ in assumption \eqref{range} is no longer needed in the case where
$M$ is independent of the first variable $s$, namely $M_s\equiv 0$.
\end{remark}
	
\begin{remark}\rm
We prove the above theorems under the a-priori boundedness assumption of $(u_n).$ This occurs in a
quite large class of problems, as Proposition \ref{bddd} shows.
\end{remark}

\begin{remark}\rm
With no additional effort, we could cover the case where an additional term $W(x)|u|^{m-2}u$ appears
in \eqref{eq} and the functional framework turns into $W^{1,p}_0(\Omega)\cap W^{1,m}_0(\Omega)$.
\end{remark}

\noindent
In the spirit of~\cite{lions}, we also get the following

\begin{corollary}\label{minsolve}
Let $N>p\geq m>1$ and assume that $\xi\mapsto L(\xi)$ is $p$-homogeneous, $\xi\mapsto M(\xi)$ is $m$-homogeneous,
$L(\xi)\geq p|\xi|^p$, $M(\xi)\geq m|\xi|^m$ and set
\begin{align}
	\label{mminbbpp}
& \S_\Omega:=\inf_{\|u\|_{L^\sigma(\Omega)}=1} \int_\Omega \frac{L(Du)}{p}+\frac{M(Du)}{m}+\frac{V(x)}{p}|u|^p,  \\
& \S_{\R^N}:=\inf_{\|u\|_{L^\sigma(\R^N)}=1} \int_{\R^N} \frac{|Du|^p}{p}+\frac{|u|^p}{p}, \notag
\end{align}
with $V(x)\to 1$ as $|x|\to\infty$. Assume furthermore that
\begin{equation}
\label{compactnesscondd}
\S_{\Omega}<\Big(\frac{\sigma-p}{\sigma -m}\Big)^{\frac{\sigma-p}{\sigma}}\S_{\R^N}.
\end{equation}
Then \eqref{mminbbpp} admits a minimizer.
\end{corollary}	

\begin{remark}\rm
We point out that, some conditions guaranteeing the nonexistence of nontrivial solutions
in the star-shaped case $\Omega=\R^N$ can be provided. For the sake of simplicity, assume that $L$ is $p$-homogeneous and that $\xi\mapsto M(s,\xi)$
is $m$-homogeneous. Then, in view of \cite[Theorem 3]{pucciser}, that holds for $C^1$ solutions by virtue
of the results of \cite{degmusqua}, we have that \eqref{eq} admits no nontrivial $C^1$ solution
well behaved at infinity, namely satisfying condition (19) of \cite{pucciser}, provided
that there exists a number $a\in\R^+$ such that a.e.\ in $\R^N$ and for all $(s,\xi)\in\R\times\R^N$
\begin{align*}
(N-p(a+1))L(\xi) &+(N-m(a+1))M(s,\xi)+(asg(s)-NG(s)) \\
&+\frac{(N-ap) V(x)+x\cdot DV(x)}{p}|s|^p-a M_s(s,\xi)s\geq 0,
\end{align*}
holding, for instance, if there exists $0\leq a\leq \frac{N-p}{p}$ such that
$$
asg(s)-NG(s)\geq 0,\,\,\,\quad
(N-ap) V(x)+x\cdot DV(x)\geq 0,\,\,\,\quad
M_s(s,\xi)s\leq 0,
$$
for a.e. $x\in\R^N$ and for all $(s,\xi)\in\R\times\R^N$. Also, in the
more particular case where $g(s)=|s|^{\sigma-2}s$ and
$V(x)=V_\infty>0$, then the above conditions simply rephrase into
$$
\sigma\geq p^*,\qquad M_s(s,\xi)s\leq 0,
$$
for every $(s,\xi)\in\R\times\R^N$. In fact, in \eqref{ilsegnos}, we consider the
opposite assumption on $M_s$.
\end{remark}

\section{Some preliminary facts}
\label{prelimsection}
\noindent
It is a standard fact that, under condition \eqref{growth0} and \eqref{ggrow}, the functionals
$$
u\mapsto \int_\Omega L(Du),\quad\,\,
u\mapsto \int_\Omega V(x)|u|^p,\quad\,\,
u\mapsto \int_\Omega G(u)
$$
are $C^1$ on $W^{1,p}_0(\Omega)\cap D^{1,m}_0(\Omega)$.
Analogously, although $M$ depends explicitly on $s$, the functional
$$
\M:W^{1,p}_0(\Omega)\cap D^{1,m}_0(\Omega)\to\R,\quad \M(u)=\int_\Omega M(u,Du),
$$
admits, thanks to condition \eqref{range}, directional
derivatives along any $v\in W^{1,p}_0(\Omega)\cap D^{1,m}_0(\Omega)$ and
$$
\M'(u)(v)=\int_\Omega M_\xi(u,Du)\cdot Dv+\int_\Omega M_s(u,Du)v,
$$
as it can be easily verified observing that $p\leq \frac{p}{p-m} \leq p^*$
and that, for $u\in W^{1,p}_0(\Omega)\cap D^{1,m}_0(\Omega)$, by Young's inequality,
for some constant $C$ it holds
\begin{align*}
 |M_\xi(u,Du)\cdot Dv| &\leq C|Du|^m+C|Dv|^m\in L^1(\Omega),\\
 |M_s(u,Du)v| &\leq C|Du|^p+C|v|^{\frac{p}{p-m}}\in L^1(\Omega).
\end{align*}
Furthermore, if $u_k\to u$ in $W^{1,p}_0(\Omega)\cap D^{1,m}_0(\Omega)$
as $k\to\infty$ then $\M'(u_k)\to \M'(u)$ in the dual space $(W^{1,p}_0(\Omega)\cap D^{1,m}_0(\Omega))^*$, as $k\to\infty$. Indeed, for
$\|v\|_{W^{1,p}_0(\Omega)\cap D^{1,m}_0(\Omega)}\leq 1$, we have
\begin{align*}
 & |\M'(u_k)(v)-\M'(u)(v)|\\ &\leq \int_\Omega| M_\xi(u_k,Du_k)-M_\xi(u,Du)|| Dv|+\int_\Omega |M_s(u_k,Du_k)-M_s(u,Du)|\,|v|\\
 & \leq \|M_\xi(u_k,Du_k)-M_\xi(u,Du)\|_{L^{m'}}\|Dv\|_{L^{m}}+\|M_s(u_k,Du_k)-M_s(u,Du)\|_{L^{p/m}}\|v\|_{L^{p/(p-m)}} \\
\noalign{\vskip4pt}
& \leq \|M_\xi(u_k,Du_k)-M_\xi(u,Du)\|_{L^{m'}}+\|M_s(u_k,Du_k)-M_s(u,Du)\|_{L^{p/m}}.
\end{align*}
This yields the desired convergence, using \eqref{growths1} and the Dominated Convergence Theorem.
Notice that the same argument carried out before applies
either to integrals defined on $\Omega$ or on $\R^N.$ Hence the following proposition is proved.
\begin{proposition}
In the hypotheses of Theorems \ref{main} and \ref{main2}, the functionals $\phi$ and $\phi_\infty$ are $C^1.$
\end{proposition}
\noindent
In addition to the assumptions on $L,M$ and $g,G$ set in the introduction,
assume now that there exist positive numbers $\delta>0$ and $\mu>p$ such that
\begin{equation}
	\label{bounddc}
\mu M(s,\xi)-M_s(s,\xi)s-M_\xi(s,\xi)\cdot\xi\geq \delta |\xi|^m,\quad
\mu L(\xi)-L_\xi(\xi)\cdot\xi\geq\delta |\xi|^p,\quad
sg(s)-\mu G(s)\geq 0,
\end{equation}
for any $s\in\R$ and all $\xi\in\R^N$.
This hypothesis is rather well established
in the framework of quasi-linear problems (cf.\ \cite{squbook})
and it allows an arbitrary Palais-Smale sequence $(u_n)$ to be bounded
in $W^{1,p}_0(\Omega)\cap D^{1,m}_0(\Omega)$, as shown in the following

\begin{proposition}\label{bddd}
	Let $j$ be as in \eqref{jdef} and assume that \eqref{range} holds.
	Let $(u_n)\subset W^{1,p}_0(\Omega)\cap D^{1,m}_0(\Omega)$ be a sequence such that
	$$
	\phi(u_n)\to c \quad \quad \phi'(u_n)\to 0 \quad \text{in $(W^{1,p}_0(\Omega)\cap D^{1,m}_0(\Omega))^*$}
	$$
	Then, if condition~\eqref{bounddc} holds, $(u_n)$ is bounded in $W^{1,p}_0(\Omega)\cap D^{1,m}_0(\Omega)$.
\end{proposition}
\begin{proof}
	Let $(w_n)\subset (W^{1,p}_0(\Omega)\cap D^{1,m}_0(\Omega))^*$ with $w_n\to 0$ as $n\to\infty$ be
	such that $\phi'(u_n)(v)=\langle w_n,v\rangle$,
	for every $v\in W^{1,p}_0(\Omega)\cap D^{1,m}_0(\Omega)$. Whence, by choosing $v=u_n$, it follows
	\begin{equation*}
		\int_\Omega j_\xi(u_n,Du_n)\cdot Du_n+\int_\Omega j_s(u_n,Du_n)u_n+\int_\Omega V(x)|u_n|^p=\langle w_n,u_n\rangle.
	\end{equation*}
	Combining this equation with $\mu \phi(u_n)=\mu c+o(1)$ as $n\to\infty$, namely
	\begin{equation*}
		\int_\Omega \mu j(u_n,Du_n)+\frac{\mu}{p}\int_\Omega V(x)|u_n|^p=\mu c+o(1),
	\end{equation*}
	recalling the definition of $j$, and using condition~\eqref{bounddc}, yields
	$$
	\frac{\mu-p}{p}\int_\Omega V(x)|u_n|^p+\delta \int_\Omega |Du_n|^p+\delta \int_\Omega |Du_n|^m\leq \mu c+\|w_n\|\|u_n\|_{W^{1,p}_0(\Omega)\cap D^{1,m}_0(\Omega)}+o(1),
	$$
	as $n\to\infty$, implying, due to $V\geq V_0$ that
	$$
	\|u_n\|_{W^{1,p}(\Omega)}^p+\|u_n\|_{D^{1,m}(\Omega)}^m\leq C+C\|u_n\|_{W^{1,p}(\Omega)}+C\|u_n\|_{D^{1,m}(\Omega)}+o(1),
	$$
	as $n\to\infty$. The assertion then follows immediately.
\end{proof}

\noindent
From now on we shall always assume to handle {\em bounded} Palais-Smale sequences,
keeping in mind that condition \eqref{bounddc} can guarantee
the boundedness of such sequences.

\begin{proposition}
	\label{convergenze}
	Let $j$ be as in \eqref{jdef} and assume that $1<m<p<N$ and $p<\sigma<p^*$.
	Let $(u_n)\subset W^{1,p}_0(\Omega)\cap D^{1,m}_0(\Omega)$ bounded be such that
	$$
	\phi(u_n)\to c \quad \quad \phi'(u_n)\to 0 \quad \text{in $(W^{1,p}_0(\Omega)\cap D^{1,m}_0(\Omega))^*$}.
	$$
	Then, up to a subsequence, $(u_n)$ converges weakly to some $u$ in $W^{1,p}_0(\Omega)\cap D^{1,m}_0(\Omega)$, $u_n(x)\to u(x)$
	and $Du_n(x)\to Du(x)$ for a.e.\ $x\in\Omega$.
\end{proposition}
\begin{proof}
	It is sufficient to justify that $Du_n(x)\to Du(x)$ for a.e.\ $x\in\Omega$. Given an arbitrary bounded subdomain
	$\omega\subset\overline{\omega}\subset\Omega$ of $\Omega$, from the fact that $\phi'(u_n)\to 0$
	in $(W^{1,p}_0(\Omega)\cap D^{1,m}_0(\Omega))^*$, we can write
	$$
	\int_\omega a(u_n,Du_n)\cdot Dv=\langle w_n,v\rangle+\langle f_n,v\rangle+
	\int_\omega v\, d\mu_n,\quad\text{for all $v\in {\mathcal D}(\omega)$},
	$$
	where $(w_n)\subset (W^{1,p}_0(\Omega)\cap D^{1,m}_0(\Omega))^*$ is vanishing, and hence in
	particular $w_n\in W^{-1,p'}(\omega)$, with $w_n\to 0$ in $W^{-1,p'}(\omega)$ as $n\to\infty$ and we have set
	\begin{align*}
		 a_n(x,s,\xi)&:=L_\xi(\xi)+M_\xi(s,\xi),\qquad\text{for all $(s,\xi)\in\R\times\R^N$}, \\
		 f_n&:=-V(x)|u_n|^{p-2}u_n+g(u_n)\in W^{-1,p'}(\omega),\qquad n\in\N,  \\
		\mu_n&:=-M_s(u_n,Du_n)\in L^1(\omega),\qquad n\in\N.
	\end{align*}
	Due to the strict convexity assumptions on the maps $\xi\mapsto L(\xi)$ and $\xi\mapsto M(s,\xi)$
	and the growth conditions on $L_\xi,M_\xi, M_s$ and $g$, all the assumptions of \cite[Theorem 1]{dalmur} are fulfilled. Precisely,
	\begin{equation*}
	|a_n(x,s,\xi)|\leq |L_\xi(\xi)|+|M_\xi(s,\xi)|\leq C|\xi|^{p-1}+C|\xi|^{m-1}\leq C+C|\xi|^{p-1},
\end{equation*}
for a.e.\ $x\in\omega$ and all $(s,\xi)\in\R\times\R^N$, and
\begin{align*}
& f_n\to f,\quad f:=-V(x)|u|^{p-2}u+g(u),\quad \text{strongly in $W^{-1,p'}(\omega)$}, \\
& \mu_n\rightharpoonup \mu,\quad \text{weakly* in ${\mathcal M}(\omega)$,\,\,\, since\,\, $\sup_{n\in\N}\|M_s(u_n,Du_n)\|_{L^1(\omega)}<+\infty$.}
\end{align*}
	 Then, it follows that
	$Du_n(x)\to Du(x)$ for a.e.\ $x\in \omega$. Finally, a simple Cantor diagonal argument allows to recover the convergence
	over the whole domain $\Omega$.
\end{proof}

\noindent
Next we prove a regularity result for the solutions of equation \eqref{eq}.
\begin{proposition}
	\label{regularityres}
	Let $j$ be as in \eqref{jdef} and assume \eqref{range} and \eqref{ilsegnos}.
	Let $u\in W^{1,p}_0(\Omega)\cap D^{1,m}_0(\Omega)$ be a solution of \eqref{eq}. Then
	$$
	u\in \bigcap_{q\geq p} L^q(\Omega),
	\quad
	\text{$u\in L^\infty(\Omega)$ and $\lim_{|x|\to\infty} u(x)=0$}.
	$$
\end{proposition}
\begin{proof}
	Let $k,i\in\N$. Then, setting $v_{k,i}(x):=(u_k(x))^i$ with $u_k(x):=\min\{u^+(x),k\}$,
	it follows that $v_{k,i}\in W^{1,p}_0(\Omega)\cap D^{1,m}_0(\Omega)$
	can be used as a test function in \eqref{eq}, yielding
	\begin{align*}
		\int_\Omega L_\xi(Du)\cdot Dv_{k,i} &+\int_\Omega M_\xi(u,Du)\cdot Dv_{k,i}  \\
			& +\int_\Omega M_s(u,Du)v_{k,i}+\int_\Omega V(x)|u|^{p-2}u v_{k,i}=\int_\Omega g(u)v_{k,i}.
	\end{align*}
Taking into account that $Dv_{k,i}$ is equal to $i u^{i-1}Du\chi_{\{0<u<k\}}$, by convexity
and positivity of the map $\xi\mapsto M(s,\xi)$ we deduce that
$M_\xi(u,Du)\cdot Dv_{k,i}\geq 0$. Moreover, by the sign condition \eqref{ilsegnos}
it follows $M_s(u,Du)v_{k,i}\geq 0$ a.e.\ in $\Omega$. Then, we reach
\begin{equation*}
	\int_\Omega i(u_k)^{i-1} L_\xi(Du_k)\cdot Du_k   +\int_\Omega V(x)|u|^{p-2}u (u_k(x))^i\leq \int_\Omega g(u)(u_k(x))^i,
\end{equation*}
yielding in turn, by \eqref{ggrow}, that for all $k,i\geq 1$
\begin{equation}
	\label{disLiu1}
	\nu i\int_\Omega (u_k)^{i-1} |Du_k|^p\leq C\int_\Omega (u^+(x))^{\sigma-1+i}.
\end{equation}
If $\hat u_k:=\min\{u^-(x),k\}$, a similar inequality
\begin{equation}
	\label{disLiu2}
	\nu i\int_\Omega (\hat u_k)^{i-1} |D\hat u_k|^p\leq C\int_\Omega (u^-(x))^{\sigma-1+i},
\end{equation}
can be obtained by using $\hat v_{k,i}:=-(\hat u_k)^i$ as a test function in \eqref{eq},
observing that by \eqref{ilsegnos},
\begin{align*}
 M_s(u,Du)\hat v_{k,i}& =-M_s(u,Du)\chi_{\{-k<u<0\}}(-u)^i\geq 0,    \\
 M_\xi(u,Du)\cdot Dv_{k,i}& =i(-u)^{i-1} \chi_{\{-k<u<0\}} M_\xi(u,Du)\cdot Du \geq 0.
\end{align*}
Once \eqref{disLiu1}-\eqref{disLiu2} are reached, the assertion follows exactly as
in \cite[Lemma 2, (a) and (b)]{Yu}.
\end{proof}	

We now recall the following version of \cite[Lemma 4.2]{degiovannilancellotti} which turns out to be a
rather useful tool in order to establish convergences in our setting. Roughly speaking, one needs some kind
of sub-criticality in the growth conditions.

\begin{lemma}\label{lemmino}
Let $\Omega\subset \R^N$ and $h: \Omega \times \R\times \R^N$ be a Carath\'eodory function, $p,m>1$, $\mu\geq 1$,
$p\leq \sigma\leq p^*$ and assume that, for every $\varepsilon>0$ there exist $ a_\varepsilon \in L^\mu(\Omega)$ such that
\begin{equation}
	\label {growtha}
|h(x,s,\xi)| \leq a_\varepsilon (x)+\varepsilon |s|^{\sigma/\mu} +\varepsilon |\xi|^{p/\mu}+\varepsilon |\xi|^{m/\mu},
\end{equation}
a.e.\ in $\Omega$ and for all $(s,\xi)\in\R\times\R^N$.
Assume that $u_n\to u$ a.e.\ in $\Omega$, $Du_n\to Du$
a.e.\ in $\Omega$ and
$$
\text{ $(u_n)$ is bounded in $W^{1,p}_0(\Omega)$,\quad
$(u_n)$ is bounded in $D^{1,m}_0(\Omega)$}.
$$
Then $h(x,u_n,Du_n)$ converges to $h(x,u,Du)$ in $L^\mu(\Omega)$.
\end{lemma}
\begin{proof}
The proof follows as in \cite[Lemma 4.2]{degiovannilancellotti} and we shall sketch it here for self-containedness.
	By Fatou's Lemma, it immediately holds that $u\in W^{1,p}_0(\Omega)\cap D^{1,m}_0(\Omega)$.
Furthermore, there exists a positive constant $C$ such that
	\begin{align*}
		|h(x,s_1,\xi_1)-h(x,s_2,\xi_2)|^\mu &\leq C(a_\eps(x))^\mu+C\eps^\mu |s_1|^\sigma+C\eps^\mu |s_2|^\sigma \\
&		+C\eps^\mu |\xi_1|^m+C\eps^\mu |\xi_2|^m+C\eps^\mu |\xi_1|^p+C\eps^\mu |\xi_2|^p,
	\end{align*}
	a.e.\ in $\Omega$ and for all $(s_1,\xi_1)\in\R\times\R^N$ and $(s_2,\xi_2)\in\R\times\R^N$.
Then, taking into account the boundedness of $(Du_n)$ in $L^p(\Omega)\cap L^m(\Omega)$
and of $(u_n)$ in $L^{\sigma}(\Omega)$ by interpolation being $p\leq \sigma\leq p^*$,
the assertion follows by applying Fatou's Lemma to the sequence
of functions $\psi_n:\Omega\to [0,+\infty]$
\begin{align*}
	\psi_n(x):=&-|h(x,u_n,Du_n)-h(x,u,Du)|^\mu + C(a_\eps(x))^\mu+C\eps^\mu |u_n|^\sigma+C\eps^\mu |u|^\sigma \\
&		+C\eps^\mu |Du_n|^m+C\eps^\mu |Du|^m+C\eps^\mu |Du_n|^p+C\eps^\mu |Du|^p,
\end{align*}
and, finally, exploiting the arbitrariness of $\eps$.
\end{proof}

\section{Proof of the result}

\subsection{Energy splitting}

The next result allows to perform an energy splitting for the functional
$$
J(u)=\int_\Omega j(u, Du),\quad u \in W^{1,p}_0(\Omega)\cap D^{1,m}_0(\Omega),
$$
along a bounded Palais-Smale sequence $(u_n)\subset W^{1,p}_0(\Omega)\cap D^{1,m}_0(\Omega)$.
The result is in the spirit of the classical Brezis-Lieb Lemma \cite{BreLieb}.

\begin{lemma}
	\label{energysplit}
Let the integrand $j$ be as in \eqref{jdef} and
$$
p-1\leq m<p-1+p/N,\qquad p\leq \sigma\leq p^*.
$$
Let $(u_n) \subset W^{1,p}_0(\Omega)\cap D^{1,m}_0(\Omega)$ with $u_n\rightharpoonup u,$
$u_n\to u$ a.e.\ in $\Omega$ and $Du_n\to Du$ a.e.\ in $\Omega$. Then
 \begin{equation} \label{integralj}
 \lim _{n \to \infty}\int_\Omega j(u_n-u, Du_n-Du)-j(u_n,D u_n)+j(u, Du)=0.
 \end{equation}
\end{lemma}
\begin{proof}
We shall apply Lemma~\ref{lemmino} to the function
$$
h(x,s,\xi) := j(s-u(x),\xi-Du(x)) - j(s,\xi),\qquad\text{for a.e.\ $x\in\Omega$ and all $(s,\xi)\in\R\times\R^N$.}
$$
Given $x\in\Omega$, $s\in\R$ and $\xi\in\R^N$,
consider the $C^1$ map $\varphi:[0,1]\to\R$ defined by setting
$$
\varphi(t):=j(s-tu(x),\xi-tDu(x)),\quad \text{for all $t\in [0,1]$}.
$$
Then, for some $\tau\in [0,1]$ depending upon $x\in\Omega$, $s\in\R$ and $\xi\in\R^N$, it holds
\begin{align*}
& h(x,s,\xi) =\varphi(1)-\varphi(0)=\varphi'(\tau)  \\
&=-j_s(s-\tau u(x),\xi-\tau Du(x))u(x)-j_\xi(s-\tau u(x),\xi-\tau Du(x))\cdot Du(x)  \\
&=-L_\xi(\xi-\tau Du(x))\cdot Du(x)  \\
&\quad -M_s(s-\tau u(x),\xi-\tau Du(x))u(x)   \\
&\quad -M_\xi(s-\tau u(x),\xi-\tau Du(x))\cdot Du(x)+G'(s-\tau u(x))u(x).
\end{align*}	
Hence, for a.e.\ $x\in\Omega$ and all $(s,\xi)\in\R\times\R^N$, it follows that
\begin{align*}
  |h(x,s,\xi)|
&\leq |L_\xi(\xi-\tau Du(x))||Du(x)|
+ |M_s(s-\tau u(x),\xi-\tau Du(x))||u(x)|   \\
&+ |M_\xi(s-\tau u(x),\xi-\tau Du(x))||Du(x)|+|G'(s-\tau u(x))||u(x)| \\
&\leq C(|\xi|^{p-1}+|Du(x)|^{p-1})|Du(x)|
+ C(|\xi|^{m}+|Du(x)|^{m})|u(x)|   \\
&+ C(|\xi|^{m-1}+|Du(x)|^{m-1})|Du(x)| +C(|s|^{\sigma-1}+|u(x)|^{\sigma-1})|u(x)|\\
&\leq \eps |\xi|^{p}+C_\eps |Du(x)|^p
+ \eps|\xi|^{p}+C_\eps |Du(x)|^{p}+C_\eps |u(x)|^{p/(p-m)}   \\
&+ \eps|\xi|^{m}+C_\eps|Du(x)|^{m}+\eps |s|^{\sigma}+C_\eps |u(x)|^{\sigma}   \\
&=a_\eps(x)+\eps |s|^{\sigma}
+\eps |\xi|^{p}+\eps|\xi|^{m},
\end{align*}
where $a_\eps:\Omega\to\R$ is defined a.e.\ by
$$
a_\eps(x):=C_\eps |Du(x)|^p+C_\eps|Du(x)|^{m}+C_\eps |u(x)|^{p/(p-m)}+C_\eps |u(x)|^{\sigma}.
$$
Notice that, as $p-1\leq m< p-1+p/N$ it holds $p\leq p/(p-m)\leq p^*$, yielding
$u\in L^{p/(p-m)}(\Omega)$ and in turn, $a_\eps \in L^1(\Omega)$.
The assertion follows directly by Lemma \ref{lemmino} with $\mu=1$.
\end{proof}

\noindent
We have the following splitting result

\begin{theorem}
	\label{energytot}
Let the integrand $j$ be as in \eqref{jdef} and
$$
p-1 \leq m \leq p-1+p/N,\qquad p<\sigma<p^*.
$$
Assume that $(u_n) \subset W^{1,p}_0(\Omega)\cap D^{1,m}_0(\Omega)$ is a bounded Palais-Smale sequence for $\phi$
at the level $c\in\R$ weakly convergent to some $u\in W^{1,p}_0(\Omega)\cap D^{1,m}_0(\Omega)$. Then
$$
\lim_{n\to\infty}\Big(\int_{\Omega}j(u_n-u,Du_n-Du)+\int_{\Omega}V_\infty\frac{|u_n-u|^p}{p}\Big)=c-\int_{\Omega}j(u,Du)-\int_{\Omega}V(x)\frac{|u|^p}{p},
$$
namely
$$
\lim_{n\to\infty} \phi_\infty(u_n-u)=c-\phi(u),
$$
being $u_n$ and $u$ regarded as elements of $W^{1,p}(\R^N)\cap D^{1,m}(\R^N)$
after extension to zero out of $\Omega$.
\end{theorem}
\begin{proof}
	In light of Proposition~\ref{convergenze}, up to a subsequence, $(u_n)$ converges
	weakly to some function $u$ in $W^{1,p}_0(\Omega)\cap D^{1,m}_0(\Omega)$, $u_n(x)\to u(x)$
	and $Du_n(x)\to Du(x)$ for a.e.\ $x\in\Omega$.
	Also, recalling that by assumption $V(x)\to V_\infty$ as $|x|\to\infty$, we have \cite{BreLieb,willembook}
	\begin{align}
	\label{Vsplit1}
	 &\lim _{n \to \infty}\int_\Omega V(x)|u_n-u|^{p}-V_\infty|u_n-u|^{p}=0,\\
	\label{Vsplit2}
	 &\lim _{n \to \infty}\int_\Omega V(x)|u_n-u|^{p}-V(x)|u_n|^{p}+V(x)|u|^p=0.
	 \end{align}
	Therefore, by virtue of Lemma~\ref{energysplit}, we conclude that
	\begin{align*}
	\lim_{n\to\infty}\phi_\infty(u_n-u)& =\lim_{n\to\infty}\Big(\int_{\Omega}j(u_n-u,Du_n-Du)+\int_{\Omega}V_\infty\frac{|u_n-u|^p}{p}\Big) \\
	& =\lim_{n\to\infty}\Big(\int_{\Omega}j(u_n-u,Du_n-Du)+\int_{\Omega}V(x)\frac{|u_n-u|^p}{p} \Big)\\
	& =\lim_{n\to\infty}\Big(\int_{\Omega}j(u_n,Du_n)+\int_{\Omega}V(x)\frac{|u_n|^p}{p}\Big)-\int_{\Omega}j(u,Du)-\int_{\Omega}V(x)\frac{|u|^p}{p} \\
	\noalign{\vskip5pt}
	& =\lim_{n\to\infty}\phi(u_n)-\phi(u)=c-\phi(u),
	\end{align*}
	concluding the proof.
\end{proof}

\begin{remark}\rm
In order to shed some light on the restriction \eqref{range} of $m$, it is readily seen that it
is a sufficient condition for the following local compactness property to hold.
Assume that $\omega$ is a smooth domain of $\R^n$ with finite measure. Then, if $(u_h)$ is a bounded sequence in
$W^{1,p}_0(\omega)$, there exists a subsequence $(u_{h_k})$ such that
$$
\text{$\Upsilon(x,u_{h_k},Du_{h_k})$ converges strongly to some $\Upsilon_0$ in $W^{-1,p'}(\omega)$},
$$
where $\Upsilon(x,s,\xi)=g(s)-M_s(s,\xi)-V(x)|s|^{p-2}s$. In fact, taking into account the growth condition on $g$ and
$M_s$, this can be proved observing that, for every $\eps>0$, there exists $C_\eps$ such that
$$
|\Upsilon(x,s,\xi)|\leq C_\eps+\eps|s|^{\frac{N(p-1)+p}{N-p}}+\eps|\xi|^{p-1+p/N},
$$
for a.e.\ $x\in\omega$ and all $(s,\xi)\in\R\times\R^N$.
\end{remark}

\subsection{Equation splitting I (super-quadratic case)}

\noindent
We shall assume that $m,p\geq 2$ and that conditions \eqref{growths1}-\eqref{growths2} hold.
The following Theorem \ref{split2-A} and the forthcoming Theorem~\ref{split2-B} (see next
subsection) are in the spirit of the Brezis-Lieb Lemma~\cite{BreLieb}, in a dual framework.
For the particular case
$$
M(s,\xi)=0\quad\text{and}\quad L(\xi)=\frac{|\xi|^p}{p},
$$
we refer the reader to~\cite{mercuriwillem}.

\begin{theorem}
	\label{split2-A}
	Assume that \eqref{range}-\eqref{jdef} hold and that
	$$
	p-1\leq m < p-1+p/N,\qquad p<\sigma<p^*.
	$$
Assume that $(u_n) \subset W^{1,p}_0(\Omega)\cap D^{1,m}_0(\Omega)$ is such that $u_n\rightharpoonup u,$
$u_n\to u$ a.e.\ in $\Omega$, $Du_n\to Du$ a.e.\ in $\Omega$ and there is $(w_n)$ in the
dual space $(W^{1,p}_0(\Omega)\cap D^{1,m}_0(\Omega))^*$ such that $w_n\to 0$ as $n\to\infty$ and,
for all $v\in W^{1,p}_0(\Omega)\cap D^{1,m}_0(\Omega)$,
\begin{equation}
		\label{laprim}
\int_\Omega j_\xi (u_n,Du_n)\cdot Dv+\int_\Omega j_s (u_n,Du_n)v+\int_\Omega V(x)|u_n|^{p-2}u_nv=\langle w_n, v\rangle.
\end{equation}
Then $\phi'(u)=0.$ Moreover, there exists a sequence $(\xi_n)$ that goes to zero in $(W^{1,p}_0(\Omega)\cap D^{1,m}_0(\Omega))^*$, such that
\begin{align}
	\label{primaconcl-rifr}
\langle\xi_n,v\rangle & :=	
\int_\Omega j_s(u_n-u,Du_n-Du)v
  + \int_\Omega j_\xi(u_n-u,Du_n-Du)\cdot Dv  \\
&  - \int_\Omega j_s(u_n,Du_n)v - \int_\Omega j_\xi(u_n,Du_n)\cdot Dv+
\int_\Omega j_s(u,Du)v + \int_\Omega j_\xi(u,Du)\cdot Dv, \notag
\end{align}
for all $v\in W^{1,p}_0(\Omega)\cap D^{1,m}_0(\Omega)$.\\
Furthermore, there exists a sequence $(\zeta_n)$ in  $(W^{1,p}_0(\Omega)\cap D^{1,m}_0(\Omega))^*$ such that
\begin{equation*}
\int_\Omega j_\xi (u_n-u,Du_n-Du)\cdot Dv  +\int_\Omega j_s (u_n-u,Du_n-Du)v
 +\int_\Omega V_\infty|u_n-u|^{p-2}(u_n-u)v=\langle \zeta_n, v\rangle
\end{equation*}
for all $v\in W^{1,p}_0(\Omega)\cap D^{1,m}_0(\Omega)$ and $\zeta_n\to 0$ as $n\to\infty$, namely $\phi_\infty'(u_n-u)\to 0$ as $n\to\infty$.
\end{theorem}

\begin{proof}
Fixed some $v\in W^{1,p}_0(\Omega)\cap D^{1,m}_0(\Omega),$
let us define for a.e.\ $x\in\Omega$ and all $(s,\xi)\in\R\times\R^N$,
\begin{align*}
f_v(x,s,\xi) &:= j_s(s-u(x),\xi-Du(x))v(x)  \\
  &+ j_\xi(s-u(x),\xi-Du(x))\cdot Dv(x)
 - j_s(s,\xi)v(x) - j_\xi(s,\xi)\cdot Dv(x).
\end{align*}
In order to prove \ref{primaconcl-rifr} we are going to show that
\begin{equation}
	\label{primaconcl}
\lim_{n\to \infty}\sup_{\|v\|_{W^{1,p}_0(\Omega)\cap D^{1,m}_0(\Omega)}\leq 1 }\Big|\int_\Omega f_v(x,u_n,Du_n)-f_v(x,u, Du) \Big|=0.
\end{equation}
As it can be easily checked, there holds
\begin{align*}
- f_v(x,s,\xi)& =\int^1_0 j_{ss}(s-\tau u(x),\xi-\tau Du(x))u(x)v(x)d\tau  \\
& +\int^1_0 j_{s\xi}(s-\tau u(x),\xi-\tau Du(x))\cdot [Du(x)v(x)+Dv(x)u(x)] d\tau  \\
& + \int^1_0[j_{\xi\xi}(s-\tau u(x),\xi-\tau Du(x))\, Du(x)]\cdot Dv (x)d\tau .
\end{align*}
Hence, by plugging the particular form of $j$ in the above equation yields
$$
-f_v(x,s,\xi)=a(x,s,\xi)v(x)+b(x,s)v(x)+c_1(x,s,\xi)\cdot Dv(x)+c_2(x,s,\xi)\cdot Dv(x)+d(x,\xi)\cdot Dv(x)
$$
where
\begin{align*}
 a(x,s,\xi)& :=\int^1_0 [M_{ss}(s-\tau u(x),\xi-\tau Du(x))u(x)
+M_{s\xi}(s-\tau u(x),\xi-\tau Du(x))\cdot Du(x)]d\tau, \\
b(x,s)       & :=-\int^1_0 G''(s-\tau u(x))u(x)d\tau ,  \\
 c_1(x,s,\xi)& :=\int^1_0 M_{\xi s}(s-\tau u(x),\xi-\tau Du(x))u(x)  d\tau,  \\
c_2(x,s,\xi) & :=\int^1_0 M_{\xi\xi}(s-\tau u(x),\xi-\tau Du(x))\, Du(x) d\tau,  \\
  d(x,\xi)   & :=\int^1_0 L_{\xi\xi}(\xi-\tau Du(x))\, Du(x) d\tau.
\end{align*}
We claim that, as $n\to\infty$, it holds
\begin{align*}
a(\cdot ,u_n,D u_n) \to  a(\cdot,u,D u)  & \qquad \text{in}\,\, L^{(p^*)'}(\Omega),  \\
b(\cdot,u_n) \to  b(\cdot,u)  & \qquad\text{in}\,\, L^{\sigma'}(\Omega), \\
c_1(\cdot,u_n,D u_n) \to  c_1(\cdot,u,D u)  & \qquad\text{in}\,\, L^{p'}(\Omega), \\
c_2(\cdot,u_n,D u_n) \to  c_2(\cdot,u,D u)  & \qquad\text{in}\,\, L^{m'}(\Omega),\\
d(\cdot,D u_n) \to  d(\cdot,D u) & \qquad\text{in}\,\, L^{p'}(\Omega).
\end{align*}
Then, using H\"older's inequality and the embeddings of
$W^{1,p}_0(\Omega)\cap D^{1,m}_0(\Omega)$ into $L^{\sigma}(\Omega)$ and $L^{p^*}(\Omega)$ we obtain
\begin{align*}
  \sup_{\|v\|_{W^{1,p}_0(\Omega)\cap D^{1,m}_0(\Omega)}\leq 1}&\Big|\int_\Omega f_v(x,u_n,Du_n)-f_v(x,u, Du) \Big|
\\& \leq C\| a(\cdot ,u_n,D u_n) - a(\cdot,u,D u)\|_{L^{(p^*)'}(\Omega)}\\ & +C\| b(\cdot ,u_n) - b(\cdot,u)\|_{L^{\sigma'}(\Omega)}, \\
 &+C\| c_1(\cdot ,u_n,D u_n) - c_1(\cdot,u,D u)\|_{L^{p'}(\Omega)},\\
 &+C\| c_2(\cdot ,u_n,D u_n) - c_2(\cdot,u,D u)\|_{L^{m'}(\Omega)},\\
 &+C\| d(\cdot ,D u_n) - d(\cdot,D u)\|_{L^{p'}(\Omega)},
\end{align*}
yielding the desired conclusion \eqref{primaconcl}. It remains to prove the convergences we claimed above. For each term, we shall exploit Lemma \ref{lemmino}. Since $m< p-1+p/N$, we can set
$$
\alpha:=\frac{m}{p^*-1},\qquad
\beta:=\frac{pN}{pN-N+p-mN}
$$
it follows $\beta>0$ and $m<m+\alpha <p$. Young's inequality yields in turn
$$
y^{(m+\alpha)/(p^*)'}\leq C y^{m/(p^*)'}+C y^{p/(p^*)'},\quad\text{for all $y\geq 0$.}
$$
Since $\beta/(p^*)'>1$ and $(m+\alpha)/(p^*)'>1$, by the growths of $M_{ss}$ and $M_{s\xi}$, we have
\begin{align*}
 |a(x ,s,\xi)|   & \leq C(|\xi|^{m}+|Du(x)|^{m})|u(x)|+C(|\xi|^{m-1}+|Du(x)|^{m-1})|Du(x)|   \\
 & \leq \eps |\xi|^{p/(p^*)'}+C_\eps |u(x)|^{\beta/(p^*)'}+C_\eps |Du(x)|^{p/(p^*)'}+\eps|\xi|^{(m+\alpha)/(p^*)'}+C_\eps |Du(x)|^{(m+\alpha)/(p^*)'} \\
&  \leq \eps |\xi|^{p/(p^*)'}+\eps |\xi|^{m/(p^*)'}+C_\eps |u(x)|^{\beta/(p^*)'}+C_\eps |Du(x)|^{p/(p^*)'}+C_\eps |Du(x)|^{m/(p^*)'}.
\end{align*}
Furthermore,
\begin{align*}
 |b(x ,s)|  &\leq C(|s|^{\sigma-2}+|u(x)|^{\sigma-2})|u(x)|
 \leq \eps |s|^{\sigma/\sigma'}+C_\eps |u|^{\sigma/\sigma'}, \\
 |c_1(x ,s,\xi)| &\leq C(|\xi|^{m-1}+|Du(x)|^{m-1})|u(x)| \\
&\leq\eps |\xi|^{p/p'}+ C_\eps|u(x)|^{p/((p-m)p')}+C_\eps |Du(x)|^{p/p'}, \\
   |c_2(x ,s,\xi)| &\leq C(|\xi|^{m-2}+|Du(x)|^{m-2})|D u(x)| \\
& \leq \eps |\xi|^{m/m'}+ C_\eps|D u(x)|^{m/m'},  \\
 |d(x, \xi)| &\leq  C(|\xi|^{p-2}+|Du(x)|^{p-2})|Du(x)|
 \leq \eps |\xi|^{p/p'}+ C_\eps|D u(x)|^{p/p'}.
\end{align*}
From the point-wise convergence of the gradients and the growth estimates of
$j_\xi,j_s$ and $g$ that $u$ is a week solutions to the problem, namely for all $v\in W^{1,p}_0(\Omega)\cap D^{1,m}_0(\Omega)$
\begin{equation}
	\label{lasec}
\int_\Omega L_\xi (Du)\cdot Dv+\int_\Omega M_\xi (u,Du)\cdot Dv+\int_\Omega M_s (u,Du)v+\int_\Omega V(x)|u|^{p-2}uv=\int_\Omega g(u)v.
\end{equation}
To get this, recall that $v\in L^{(p/m)'}(\Omega)$ and the sequence $(M_s (u_n,Du_n))$ is bounded in $L^{p/m}(\Omega)$ and
hence it converges weakly to $M_s (u,Du)$ in $L^{p/m}(\Omega)$.
Thanks to Proposition \ref{regularityres} (recall that $\beta\geq p$
if and only if $m\geq p-2+p/N$ and this is the case since $m\geq p-1$), we have $L^{\beta}(\Omega)$. Hence,
$$
u\in L^\sigma(\Omega)\cap L^{\frac{p}{p-m}}(\Omega)\cap L^{\beta}(\Omega),
$$
being $p\leq p/(p-m)<p^*$ and $p<\sigma<p^*$. By the previous inequalities the claim follows
by Lemma \ref{lemmino} with the choice $\mu=(p^*)',\sigma',p',m'$ and $p'$ respectively.
Let us now recall a dual version of properties \eqref{Vsplit1}-\eqref{Vsplit2} (cf.\ \cite{willembook}),
namely there exist two sequences $(\mu_n)$ and $(\nu_n)$ in $(W^{1,p}_0(\Omega)\cap D^{1,m}_0(\Omega))^*$
which converge to zero as $n\to\infty$ and such that
\begin{align*}
\int_\Omega V_\infty|u_n-u|^{p-2}(u_n-u)v &=\int_\Omega V(x)|u_n-u|^{p-2}(u_n-u)v+\langle \nu_n, v\rangle,  \\
\int_\Omega V(x)|u_n-u|^{p-2}(u_n-u)v &=\int_\Omega V(x)|u_n|^{p-2}u_nv-\int_\Omega V(x)|u|^{p-2}uv+\langle \mu_n, v\rangle,
\end{align*}
for every $v\in W^{1,p}_0(\Omega)\cap D^{1,m}_0(\Omega)$.
Whence, by collecting \eqref{laprim}, \eqref{primaconcl-rifr}, \eqref{primaconcl}, \eqref{lasec}, we get
\begin{align*}
& \int_\Omega j_\xi (u_n-u,Du_n-Du)\cdot Dv+\int_\Omega j_s (u_n-u,Du_n-Du)v
 +\int_\Omega V_\infty|u_n-u|^{p-2}(u_n-u)v \\
& =\int_\Omega j_\xi (u_n,Du_n)\cdot Dv+\int_\Omega j_s (u_n,Du_n)v +\int_\Omega V(x)|u_n|^{p-2}u_nv \\
&-\int_\Omega j_\xi (u,Du)\cdot Dv-\int_\Omega j_s (u,Du)v -\int_\Omega V(x)|u|^{p-2}uv
+\langle \xi_n+\mu_n+\nu_n, v\rangle =\langle \zeta_n, v\rangle,\qquad
\end{align*}
where $\langle \zeta_n, v\rangle:=\langle w_n+ \xi_n+\mu_n+\nu_n, v\rangle$
and $\zeta_n\to 0$ as $n\to\infty$. This concludes the proof.
\end{proof}

\subsection{Equation splitting II (sub-quadratic case)}
\label{subquadratics}

We assume that \eqref{ass-sub1}-\eqref{ass-sub3} hold.

\begin{theorem}
	\label{split2-B}
	Assume \eqref{ilsegnos}, let the integrand $j$ be as in \eqref{jdef} and $p\leq 2$ or $m\leq 2$ or $\sigma\leq 2$,
	$$
	p-1 \leq m < p-1+p/N,\qquad p<\sigma<p^*.
	$$
Assume that $(u_n) \subset W^{1,p}_0(\Omega)\cap D^{1,m}_0(\Omega)$ is such that $u_n\rightharpoonup u,$
$u_n\to u$ a.e.\ in $\Omega$, $Du_n\to Du$ a.e.\ in $\Omega$ and there exists $(w_n)$ in  $(W^{1,p}_0(\Omega)\cap D^{1,m}_0(\Omega))^*$ such that $w_n\to 0$ as $n\to\infty$ and,
for every $v\in W^{1,p}_0(\Omega)\cap D^{1,m}_0(\Omega)$,
\begin{equation*}
\int_\Omega j_\xi (u_n,Du_n)\cdot Dv+\int_\Omega j_s (u_n,Du_n)v+\int_\Omega V(x)|u_n|^{p-2}u_nv=\langle w_n, v\rangle.
\end{equation*}
Then $\phi'(u)=0.$ Moreover, there exists a sequence $(\hat\xi_n)$ that goes to zero in $(W^{1,p}_0(\Omega)\cap D^{1,m}_0(\Omega))^*$,  such that
\begin{align}
	\label{primaconcl-rifr2}
\langle\hat\xi_n,v\rangle & :=	
\int_\Omega j_s(u_n-u,Du_n-Du)v
  + \int_\Omega j_\xi(u_n-u,Du_n-Du)\cdot Dv  \\
&  - \int_\Omega j_s(u_n,Du_n)v - \int_\Omega j_\xi(u_n,Du_n)\cdot Dv+
\int_\Omega j_s(u,Du)v + \int_\Omega j_\xi(u,Du)\cdot Dv, \notag
\end{align}
for all $v\in W^{1,p}_0(\Omega)\cap D^{1,m}_0(\Omega)$.\\
Furthermore, there exists a sequence $(\hat\zeta_n)$ in $W^{1,p}_0(\Omega)\cap D^{1,m}_0(\Omega)$ with
\begin{equation*}
\int_\Omega j_\xi (u_n-u,Du_n-Du)\cdot Dv  +\int_\Omega j_s (u_n-u,Du_n-Du)v
 +\!\int_\Omega V_\infty|u_n-u|^{p-2}(u_n-u)v=\langle \hat \zeta_n, v\rangle
\end{equation*}
for all $v\in W^{1,p}_0(\Omega)\cap D^{1,m}_0(\Omega)$ and $\hat\zeta_n\to 0$ as $n\to\infty$, namely $\phi_\infty'(u_n-u)\to 0$ as $n\to\infty$.
\end{theorem}
\begin{proof}
Keeping in mind the argument in proof of Theorem~\ref{split2-A}, here we shall be more sketchy.
For every $s\in\R$ and $\xi\in\R^N$ we plug $L,M,G$ into the equation
	\begin{align*}
	f_v(x,s,\xi) &= j_s(s-u(x),\xi-Du(x))v(x) \\ &  + j_\xi(s-u(x),\xi-Du(x))\cdot Dv(x) -
	j_s(s,\xi)v(x) - j_\xi(s,\xi)\cdot Dv(x),
	\end{align*}
	thus obtaining
	\begin{align*}
	f_v(x,s,\xi) &=(M_s(s-u(x),\xi-Du(x))-M_s(s,\xi))v(x)-(G'(s-u(x))-G'(s))v(x)\\
	&+(M_\xi(s-u(x),\xi-Du(x))-M_\xi(s,\xi))\cdot Dv(x)+(L_\xi(\xi-Du(x))-L_\xi(\xi))\cdot Dv(x)\\
	&=a'v(x)+b'v(x)+c'\cdot Dv(x)+d'\cdot Dv(x).
	\end{align*}
	We write the term $M_\xi(s-u(x),\xi-Du(x))-M_\xi(s,\xi)$ in a more suitable form, namely
	\begin{align*}
	c'&=M_\xi(s-u(x),\xi-Du(x))-M_\xi(s,\xi) \\
	&= \underbrace{M_\xi(s-u(x),\xi-Du(x))-M_\xi(s,\xi-Du(x))}_{c'_1(x,s,\xi)}
	+ \underbrace{M_\xi(s,\xi-Du(x))-M_\xi(s,\xi)}_{c'_2(x,s,\xi)},
	\end{align*}
	so that
	$$
	f_v(x,s,\xi)=a'(x,s,\xi)v(x)+b'(x,s)v(x)+(c'_1(x,s,\xi)+c'_2(x,s,\xi))\cdot Dv(x)+d'(x,\xi)\cdot Dv(x).
	$$
	The term $a'$ admits the same growth condition of $a$, cf.\ the
	proof of Theorem~\ref{split2-A}. Also, since
	$$
	c'_1(x,s,\xi) =-\int^1_0 M_{\xi s}(s-\tau u(x),\xi- Du(x))u(x)  d\tau,
	$$
	as for the term $c_1$ in the proof of Theorem~\ref{split2-A} we obtain
	$$
	|c_1'(x,s,\xi)|\leq\eps |\xi|^{p/p'}+ C_\eps|u(x)|^{p/((p-m)p')}+C_\eps |Du(x)|^{p/p'}.
	$$
On the other hand, directly from assumptions \eqref{ass-sub1}-\eqref{ass-sub3} we get
	\begin{equation*}
	|b'(x ,s)| \leq C |u(x)|^{\sigma/\sigma'},\quad
	|c'_2(x ,s,\xi)| \leq C|D u(x)|^{m/m'},\quad
	|d'(x, \xi)| \leq  C |D u(x)|^{p/p'}.
	\end{equation*}
	The conclusion follows then by the same argument carried out in Theorem \ref{split2-A}.
\end{proof}

\noindent
In the spirit of \cite[Lemma 8.3]{willembook}, we have the following

\begin{lemma}\label{likebook}
Under the hypotheses of Theorem \ref{main} or \ref{main2}, let $(y_n)\subset\R^N$ with $|y_n|\rightarrow \infty,$
\begin{align*}
&u_n(\cdot+y_n)\rightharpoonup u \qquad \text{in $W^{1,p}(\R^N)\cap D^{1,m}(\R^N)$},\\
&u_n(\cdot+y_n)\rightarrow u \qquad \text{a.e.\ in $\R^N$},\\
&Du_n(\cdot+y_n)\rightarrow Du \qquad \text{a.e.\ in $\R^N$},\\
&\phi_\infty(u_n)\rightarrow c,\\
&\phi'_\infty(u_n) \rightarrow 0 \qquad \text{in $(W^{1,p}_0(\Omega)\cap D^{1,m}_0(\Omega))^*$}.
\end{align*}
Then $\phi'_\infty(u)=0$ and, setting $v_n:=u_n-u(\cdot-y_n)$, we have
\begin{align}
	\label{decomp-energy}
&\phi_\infty(v_n)\rightarrow c-\phi_\infty(u)\\
\label{gradientvanish}
&\phi_\infty'(v_n)\rightarrow 0\quad \text{in $(W^{1,p}_0(\Omega)\cap D^{1,m}_0(\Omega))^*$},
\end{align}
and $\|v_n\|_p^p=\|u_n\|_p^p-\|u\|_p^p+o(1)$ and $\|v_n\|_m^m=\|u_n\|_m^m-\|u\|_m^m+o(1)$ as $n\to\infty$.
\end{lemma}

\begin{proof}
The energy splitting \eqref{decomp-energy} follows by Theorem~\ref{energytot} applied
with $\Omega=\R^N$ and the sequence $(u_n)$ replaced by $(u_n(\cdot+y_n))$.
Take now $\varphi\in \mathcal D(\Omega)$
with $\|\varphi\|_{W^{1,p}_0(\Omega)\cap D^{1,m}_0(\Omega)}\leq 1$ and
define $\varphi_n:=\varphi(\cdot+y_n)$. Then
$\varphi_n\in {\mathcal D}(\Omega_n)$, where $\Omega_n=\Omega-\{y_n\}\subset\Omega$ for $n$ large.
For any $n\in\N$, we get
$$
\langle\phi_\infty'(v_n),\varphi\rangle=\langle\phi_\infty'(u_n(\cdot+ y_n)-u),\varphi_n\rangle.
$$
By the splitting argument in the proof of Theorem \ref{split2-A}, it follows that
$$
\langle\phi_\infty'(u_n(\cdot+y_n)-u),\varphi_n\rangle=\langle\phi_\infty'(u_n(\cdot+y_n)),
\varphi_n\rangle-\langle\phi_\infty'(u),\varphi_n\rangle+\langle \zeta_n,\varphi_n\rangle,
$$
where $\zeta_n\to 0$ in the dual of $W^{1,p}_0(\Omega)\cap D^{1,m}_0(\Omega)$.
If we prove that $u$ is critical for $\phi_\infty$, then the right-hand side reads as
$\langle\phi_\infty'(u_n),\varphi\rangle+\langle \zeta_n,\varphi_n\rangle,$
and also the second limit \eqref{gradientvanish} follows.
To prove that $\phi'_\infty(u)=0$ we observe that, for all $\varphi$ in ${\mathcal D}(\R^N)$,
$$
\langle\phi_\infty'(u_n(\cdot+y_n)),\varphi\rangle\rightarrow \langle\phi_\infty'(u),\varphi\rangle,\quad
|\langle\phi_\infty'(u_n(\cdot+y_n)),\varphi\rangle|\leq \|\phi_\infty'(u_n)\|_*\|\varphi\|_{W^{1,p}_0(\Omega)\cap D^{1,m}_0(\Omega)}\to 0.
$$
Indeed, defining $\hat\varphi_n:=\varphi(\cdot-y_n)$, since  $|y_n|\rightarrow \infty$ as $n\to\infty$,
we have ${\rm supp} \,\hat\varphi_n \subset \Omega,$ for $n$ large enough and
$\|\hat\varphi_n\|_{W^{1,p}_0(\Omega)\cap D^{1,m}_0(\Omega)}
=\|\varphi\|_{{W^{1,p}(\R^N)\cap D^{1,m}(\R^N)}}$.
The last assertion follows by using Brezis-Lieb Lemma \cite{BreLieb}.
\end{proof}

We can finally come to the proof of the main results.

\section{Proof of Theorems~\ref{main} and~\ref{main2} completed}
We follow the scheme of the proof given in \cite[p.121]{willembook}. Let $(u_n)\subset W^{1,p}_0(\Omega)\cap D^{1,m}_0(\Omega)$
be a bounded Palais-Smale sequence for $\phi$ at the level $c\in\R$. Hence,
there exists a sequence $(w_n)$ in the dual of
$W^{1,p}_0(\Omega)\cap D^{1,m}_0(\Omega)$ such that $w_n\to 0$ and $\phi(u_n)\to c$ as $n\to\infty$ and,
	for all $v\in W^{1,p}_0(\Omega)\cap D^{1,m}_0(\Omega)$, we have
	\begin{align*}
	\int_\Omega L_\xi (Du_n)\cdot Dv&+\int_\Omega M_\xi (u_n,Du_n)\cdot Dv+\int_\Omega M_s (u_n,Du_n)v \\
	& +\int_\Omega V(x)|u_n|^{p-2}u_nv=\int_\Omega g(u_n)v+\langle w_n, v\rangle.
	\end{align*}
	Since $(u_n)$ is bounded in $W^{1,p}_0(\Omega)\cap D^{1,m}_0(\Omega)$, up to a subsequence, it converges weakly to some function $v_0\in W^{1,p}_0(\Omega)\cap D^{1,m}_0(\Omega)$
	and, by virtue of Proposition~\ref{convergenze}, $(u_n)$ and $(Du_n)$ converge to $v_0$ and $Dv_0$ a.e.\ in $\Omega$, respectively.
	In turn (see also the proof of Theorem~\ref{split2-A}) it follows
	\begin{equation*}
	\int_\Omega L_\xi (Dv_0)\cdot Dv+\int_\Omega M_\xi (v_0,Dv_0)\cdot Dv+\int_\Omega M_s (v_0,Dv_0)v
	+\int_\Omega V(x)|v_0|^{p-2}v_0v=\int_\Omega g(v_0)v,
	\end{equation*}
	for any $v\in W^{1,p}_0(\Omega)\cap D^{1,m}_0(\Omega)$.
	By combining Theorem~\ref{energytot} and Theorem~\ref{split2-A}, setting $u_n^1:=u_n-v_0$
	and thinking the functions on $\R^N$ after extension to zero out of $\Omega$, get
	\begin{align}
	&	\quad \phi_\infty(u_n^1)\to c-\phi(v_0),\quad n\to\infty,   \label{ppripropp} \\
	\int_{\R^N} L_\xi (Du_n^1)\cdot Dv&+\int_{\R^N} M_\xi (u_n^1,Du_n^1)\cdot Dv+\int_{\R^N} M_s (u_n^1,Du_n^1)v    \label{secpropp}  \\
	& +\int_{\R^N} V_\infty|u_n^1|^{p-2}u_n^1v=\int_{\R^N} g(u_n^1)v+\langle w_n^1, v\rangle. \notag
	\end{align}
	where $(w_n^1)$ is a sequence in the dual of $W^{1,p}_0(\Omega)\cap D^{1,m}_0(\Omega)$ with $w_n^1\to 0$ as $n\to\infty$.
	In turn, it follows that $(u_n^1)$ is Palais-Smale sequence for $\phi_\infty$ at the energy level $c-\phi(v_0)$. In addition,
	$$
	\|u_n^1\|_p^p= \|u_n\|_p^p-\|v_0\|_p^p+o(1),\qquad
	\|u_n^1\|_m^m= \|u_n\|_m^m-\|v_0\|_m^m+o(1),\quad\text{as $n\to\infty$,}
	$$
	by the Brezis-Lieb Lemma \cite{BreLieb}. Let us now define
 $$
 \varpi:=\limsup_{n\to \infty}\sup_{y\in \R^N}\int_{B(y,1)}|u^1_n|^{p}.
 $$
If it is the case that $\varpi=0$, then, according to \cite[Lemma I.1]{lions}, $(u_n^1)$
converges to zero in $L^r(\R^N)$ for every $r\in (p, p^*).$
Then, one obtains that
$$
\lim_{n\to\infty}\int_\Omega g(u_n^1)u_n^1=0,\qquad
\int_\Omega M_s (u_n^1,Du_n^1)u_n^1\geq 0,
$$
where the inequality follows by the sign condition \eqref{ilsegnos}. In turn,
testing equation \eqref{secpropp} with $v=u_n^1$, by the coercivity and convexity of $\xi\mapsto L(\xi), M(s,\xi)$, we have
\begin{align*}
	&\limsup_{n\to\infty}\Big[\nu\int_{\R^N} |Du_n^1|^p+\nu\int_{\R^N} |Du_n^1|^m
	+V_\infty\int_{\R^N} |u_n^1|^{p}\Big] \\
	&\leq\limsup_{n\to\infty}\Big[\int_{\R^N} L_\xi (Du_n^1)\cdot Du_n^1+\int_{\R^N} M_\xi (u_n^1,Du_n^1)\cdot Du_n^1
+\int_{\R^N} V_\infty|u_n^1|^{p}\Big]\leq 0,
\end{align*}
yielding that $(u_n^1)$ strongly converges to zero in $W^{1,p}(\R^N)\cap D^{1,m}(\R^N)$,
concluding the proof in this case. If, on the contrary, it holds $\varpi>0$, then, there exists an unbounded sequence
$(y_n^1)\subset\R^N$ with $\int_{B(y_n^1,1)}|u^1_n|^p>\varpi/2$. Whence, let us consider $v_n^1:=u_n^1(\cdot+y_n^1)$,
which, up to a subsequence, converges weakly and pointwise to some  $v_1\in W^{1,p}(\R^N)\cap D^{1,m}(\R^N)$,
which is nontrivial, due to the inequality $\int_{B(0,1)}|v_1|^p\geq \varpi/2$. Notice that, of course,
\begin{equation*}
\lim_{n\to\infty}\phi_\infty(v_n^1)=\lim_{n\to\infty}\phi_\infty(u_n^1)=c-\phi(v_0).
\end{equation*}
Moreover, since $|y_n^1|\to\infty$ and $\Omega$ is an exterior domain, for all $\varphi\in {\mathcal D}(\R^N)$
we have $\varphi(\cdot-y_n^1)\in{\mathcal D}(\Omega)$ for $n\in\N$ large enough.
Whence, in light of equation~\eqref{secpropp}, for every $\varphi\in {\mathcal D}(\R^N)$ we get
\begin{align*}
& \int_{\R^N} L_\xi (Dv_n^1)\cdot D\varphi+\int_{\R^N} M_\xi (v_n^1,Dv_n^1)\cdot D\varphi
+\int_{\R^N} M_s (v_n^1,Dv_n^1)\varphi \\
& +\int_{\R^N} V_\infty|v_n^1|^{p-2}(v_n^1)\varphi
 -\int_{\R^N} g(v_n^1)\varphi =
 \int_{\R^N} L_\xi (Du_n^1)\cdot D\varphi(\cdot-y_n^1)  \\
&+\int_{\R^N} M_\xi (u_n^1,Du_n^1)\cdot D\varphi(\cdot-y_n^1)
+\int_{\R^N} M_s (u_n^1,Du_n^1)\varphi(\cdot-y_n^1)
 +\int_{\R^N} V_\infty|u_n^1|^{p-2}(u_n^1)\varphi(\cdot-y_n^1) \\
& -\int_{\R^N} g(u_n^1)\varphi(\cdot-y_n^1)=\langle w_n^1, \varphi(\cdot+y_n^1)\rangle.
\end{align*}
Defining the form $\langle \hat w_n^1, \varphi\rangle:=\langle w_n^1, \varphi(\cdot-y_n^1)\rangle$
for all $\varphi\in {\mathcal D}(\R^N)$, we conclude that
\begin{align*}
\int_{\R^N} L_\xi (Dv_n^1)\cdot D\varphi &+\int_{\R^N} M_\xi (v_n^1,Dv_n^1)\cdot D\varphi
+\int_{\R^N} M_s (v_n^1,Dv_n^1)\varphi \\
& +\int_{\R^N} V_\infty|v_n^1|^{p-2}(v_n^1)\varphi
 -\int_{\R^N} g(v_n^1)\varphi =\langle \hat w_n^1,\varphi\rangle,\quad\forall  \varphi\in {\mathcal D}(\R^N).
\end{align*}
Since $(\hat w_n^1)$ converges to zero in the dual of $W^{1,p}(\R^N)\cap D^{1,m}(\R^N)$,
it follows by Proposition~\ref{convergenze} (with $V=V_\infty$ and $\Omega=\R^N$) that the gradients
$Dv_n^1$ converge point-wise to $Dv_1$, namely
\begin{equation}
	\label{Dquasiov1}
Dv_n^1(x)\to Dv_1(x),\qquad\text{a.e.\ in $\R^N$.}
\end{equation}
Setting $u_n^2:=u_n^1-v_1(\cdot-y_n^1)$, in light of \eqref{ppripropp}-\eqref{secpropp} and \eqref{Dquasiov1},
we can apply Lemma \ref{likebook} to the sequence $(v_n^1)$, getting
\begin{equation*}
\lim_{n\to\infty}\phi_\infty(u_n^2)=c-\phi(v_0)-\phi_\infty(v_1),
\end{equation*}
as well as $\phi_\infty(v_1)=0$ and, furthermore, for every $v\in W^{1,p}_0(\Omega)\cap D^{1,m}_0(\Omega)$, we have
\begin{align*}
&	\int_{\R^N} L_\xi (Du_n^2)\cdot Dv+\int_{\R^N} M_\xi (u_n^2,Du_n^2)\cdot Dv+\int_{\R^N} M_s (u_n^2,Du_n^2)v \\
&	+\int_{\R^N} V_\infty|u_n^2|^{p-2}u_n^2v-\int_{\R^N} g(u_n^2)v =\langle \zeta_n^2, v\rangle,
\end{align*}
where $(\zeta_n^2)$ goes to zero in the dual of $W^{1,p}_0(\Omega)\cap D^{1,m}_0(\Omega)$.
In turn, $(u_n^2)\subset W^{1,p}(\R^N)\cap D^{1,m}(\R^N)$ is a
Palais-Smale sequence for $\phi_\infty$ at the energy level $c-\phi(v_0)-\phi(v_1)$.
Arguing on $(u_n^2)$ as it was done for $(u_n^1)$, either $u_n^2$ goes to zero strongly
in $W^{1,p}(\R^N)\cap D^{1,m}(\R^N)$ or we can generate a new $(u_n^3)$.
By iterating the above procedure, one obtains diverging sequences $(y_n^i)$, $i=1,\dots,k-1$, solutions
$v_i$ on $\R^N$ to the limiting problem, $i=1,\dots,k-1$ and a sequence
$$
u^k_n=u_n-v_0-v_1(\cdot-y_n^1)-v_2(\cdot-y_n^2)-\cdots-v_{k-1}(\cdot-y_n^{k-1}),
$$
such that (recall again Lemma \ref{likebook}) as $n\to\infty$
\begin{align}
\label{prima-norme}
 \|u^k_n\|^p_p & = \|u_n\|^p_p -\|v_0\|^p_p- \|v_1\|^p_p- \cdots - \|v_{k-1}\|^p_p+o(1), \\
 \|u^k_n\|^m_m & = \|u_n\|^m_m -\|v_0\|^m_m- \|v_1\|^m_m-\cdots- \|v_{k-1}\|^m_m+o(1), \notag
\end{align}
as well as $\phi_\infty'(u^k_n)\to 0$ in $(W^{1,p}_0(\Omega)\cap D^{1,m}_0(\Omega))^*$ and
\begin{equation*}
 \phi_\infty(u^k_n)\to c- \phi(v_0)-\sum_{j=1}^{k-1}\phi_\infty(v_j).
\end{equation*}
Notice that the iteration is forced to end up after a
finite number $k\geq 1$ of steps. Indeed, for every nontrivial critical point $v\in W^{1,p}(\R^N)\cap D^{1,m}(\R^N)$ of
$\phi_\infty$ we have,
\begin{equation*}
\int_{\R^N} L_\xi (Dv)\cdot Dv+\int_{\R^N} M_\xi (v,Dv)\cdot Dv+\int_{\R^N} M_s (v,Dv)v
+\int_{\R^N} V_\infty|v|^{p}=\int_{\R^N} g(v)v,
\end{equation*}
yielding by the sign condition, the coercivity-convexity conditions and the growth of $g$,
\begin{equation}
	\label{bddbeloww}
\min\{\nu,V_\infty\}\|v\|_p^p+\|Dv\|_{L^m(\R^N)}^m\leq C_g\|v\|_{L^\sigma(\R^N)}^\sigma\leq C_g S_{p,\sigma}\|v\|_{p}^\sigma,
\end{equation}
so that, due to $\sigma>p$, it holds
\begin{equation}
	\label{boundbelowww}
\|v\|_p^p\geq\left[\frac{\min\{\nu, V_\infty\}}{C_g S_{p,\sigma}}\right]^{\frac{p}{\sigma-p}}=:\Gamma_\infty>0,
\end{equation}
thus yielding from \eqref{prima-norme}
$$
\|u^k_n\|^p_p\leq\|u_n\|^p_p-\|v_0\|^p_p-(k-1)\Gamma_\infty+o(1).
$$
By boundedness of $(u_n),$ $k$ has to be finite.
Hence $u^k_n\rightarrow0$ strongly in  $W^{1,p}(\R^N)\cap D^{1,m}(\R^N)$ at some finite index $k\in\N$.
This concludes the proof.
\qed
\smallskip

\section{Proof of Corollary~\ref{compat}}
As a byproduct of the proof of the Theorems~\ref{main} and~\ref{main2},
since the $p$ norm is bounded away from zero on the set of nontrivial critical points of $\phi_\infty,$ cf.\ \eqref{bddbeloww},we can
estimate $\phi_\infty$ from below on that set. In order to do so, we
use condition \eqref{bounddc}. For any nontrivial critical point
of the functional $\phi_\infty$, we have (see the proof of Proposition~\ref{bddd})
\begin{align*}
\mu\phi_\infty(v)\geq \delta \int_\Omega |Dv|^p+\frac{\mu-p}{p}V_\infty\int_{\R^N}|v|^p\geq \min\left\{\delta,\frac{\mu-p}{p}V_\infty\right\}\|v\|^p_p.
\end{align*}
An analogous argument applies to $\phi,$  yielding for any nontrivial critical point
\begin{align*}
\mu\phi(u)\geq \delta \int_\Omega |Du|^p+\frac{\mu-p}{p}V_0\int_{\Omega}|u|^p\geq \min\left\{\delta,\frac{\mu-p}{p}V_0\right\}\|u\|^p_p.
\end{align*}
Now notice that, recalling \eqref{boundbelowww} and a similar variant for the norm of the critical points of $\phi$ in place of $\phi_\infty$, setting also
$$
e_\infty:=\min\left\{\frac{\delta}{\mu},\frac{\mu-p}{\mu p}V_\infty\right\}\Gamma_\infty,\quad
e_0:=\min\left\{\frac{\delta}{\mu},\frac{\mu-p}{\mu p}V_0\right\}\Gamma_0,\quad
\Gamma_0:=\left[\frac{\min\{\nu, V_0\}}{C_g S_{p,\sigma}}\right]^{\frac{p}{\sigma-p}}>0,
$$
from Theorems \ref{main} or \ref{main2} we have
$c\geq \ell e_0+k e_\infty$
for some $\ell\in\{0,1\}$ and non-negative integer $k.$  Condition $c<c^*:=e_\infty$
implies necessarily $k<1$, namely $k=0$. This provides the desired compactness result, using Theorems \ref{main} or \ref{main2}.
\qed

\section{Proof of Corollary~\ref{minsolve}}
Defining the functionals $J,Q:W^{1,p}_0(\Omega)\cap D^{1,m}_0(\Omega)\to\R$ by
$$
J(u):=\frac{1}{p}\int_{\Omega}L(Du)+\frac{1}{m}\int_{\Omega}M(Du)+\frac{1}{p}\int_{\Omega}V(x)|u|^p,\qquad Q(u):=\frac{\S_\Omega}{\sigma}\int_{\Omega}|u|^\sigma,
$$
and given a minimization sequence $(u_n)$ for problem \eqref{mminbbpp}, by Ekeland's variational principle,
without loss of generality we can replace it by a new minimization sequence, still denoted by $(u_n)$ for which
there exists a sequence $(\lambda_n)\subset \R$ such that for all $v\in W^{1,p}_0(\Omega)\cap D^{1,m}_0(\Omega)$
$$
J'(u_n)(v)-\lambda_n Q'(u_n)(v)=\langle w_n,v\rangle,\quad \text{with $w_n\to 0$ in the dual of $W^{1,p}_0(\Omega)\cap D^{1,m}_0(\Omega)$}.
$$
Taking into account the homogeneity of $L$ and $M$, choosing $v=u_n$ this means
$$
\int_{\Omega}L(Du_n)+\int_{\Omega}M(Du_n)+\int_{\Omega}V(x)|u_n|^p-\S_\Omega\lambda_n\int_{\Omega}|u_n|^\sigma=\langle w_n,u_n\rangle.
$$
Up to a subsequence, let us define 
$$
\lambda:=\frac{1}{\S_\Omega}\lim_n \int_{\Omega}L(Du_n)+M(Du_n)+V(x)|u_n|^p.
$$
Since $\|u_n\|_{L^\sigma(\Omega)=1}$ for all $n$ and 
$\int_{\Omega}L(Du_n)/p+M(Du_n)/m+V(x)|u_n|^p/p\to \S_\Omega$ as $n\to\infty$, this means that $\lambda\in [m,p]$ and $(u_n)$
is a Palais-Smale sequence for the functional $I(u):=J(u)-\lambda Q(u)$ at an energy level
\begin{equation}
	\label{levelcc}
c\leq \frac{\sigma-m}{\sigma }\,\S_\Omega ,
\end{equation}
since it holds (recall that $p\geq m$), as $n\to\infty$,
\begin{align*}
I(u_n) &=\frac{1}{p}\int_{\Omega}L(Du_n)+\frac{1}{m}\int_{\Omega}M(Du_n)+\frac{1}{p}\int_{\Omega}V(x)|u_n|^p-\frac{\S_\Omega}{\sigma} \lambda \\
& \leq \S_\Omega-\frac{\S_\Omega}{\sigma}\lambda +o(1)\leq \S_\Omega \frac{\sigma-m}{\sigma}+o(1).
\end{align*}
From Corollary~\ref{compat} (applied with $L(Du)$ replaced by $L(Du)/p$,
$M(u,Du)$ replaced by $M(Du)/m$ and $G(u)=\frac{\S_\Omega}{\sigma} \lambda|s|^{\sigma}$), the compactness of
$(u_n)$ holds provided (in the notations of Corollary~\ref{compat})
$$
c<\min\left\{\frac{\delta}{\mu},\frac{\mu-p}{\mu p}V_\infty\right\}\left[\frac{\min\{\nu, V_\infty\}}{C_g S_{p,\sigma}}\right]^{\frac{p}{\sigma-p}}.
$$
In our case, we can take
$\mu=\sigma$,
$\delta=\frac{\sigma-p}{p}$,
$C_g=p\S_\Omega$,
$V_\infty=1$,
$\nu=1$,
$S_{p,\sigma}=(p\S_{\R^N})^{-\sigma/p}$,
yielding
$$
c<\frac{\sigma-p}{\sigma }\,{\S_{\R^N}^{\frac{\sigma}{\sigma-p}}}/
{\S_{\Omega}^{\frac{p}{\sigma-p}}}.
$$
Hence, finally, by combining this conclusion with \eqref{levelcc} the compactness (and
in turn the solvability of the minimization problem) holds if \eqref{compactnesscondd} holds,
concluding the proof.

\bigskip

\end{document}